\documentclass[12pt,dvipsnames]{article}
\usepackage{a4} 

\usepackage{amsthm,amsmath,amssymb,amsfonts}
\usepackage{mathtools}

\usepackage{xcolor}
\usepackage{graphics}
\usepackage{graphicx}
\usepackage{hyperref}
\usepackage{float}
\usepackage{booktabs, multirow}
\usepackage{url}
\usepackage{verbatim,upquote}
\usepackage[section]{algorithm}
\usepackage{sectsty}
\sectionfont{\large}
\subsectionfont{\normalsize}
\usepackage{nicematrix, tikz}

\hypersetup{colorlinks,urlcolor=blue,citecolor=hotpink,linkcolor=blue}

\graphicspath{{figs/}}

\newtheorem{theorem}{Theorem}
\newtheorem*{theorem*}{Theorem}
\newtheorem{lemma}[theorem]{Lemma}
\newtheorem{corollary}[theorem]{Corollary}

\numberwithin{equation}{section}
\numberwithin{table}{section}
\numberwithin{figure}{section}
\numberwithin{theorem}{section}

\definecolor{hotpink}{rgb}{0.9,0,0.5}
\hypersetup{colorlinks,urlcolor=blue,citecolor=hotpink,linkcolor=blue}

\def\>{\mskip\medmuskip}

\newcommand{\C}{\ensuremath{\mathbb{C}}}

\newcommand{\floor}[1]{\lfloor #1 \rfloor}
\newcommand{\ceil}[1]{\lceil #1 \rceil}

\DeclareMathOperator{\diag}{diag}

\DeclareMathOperator{\adj}{\operatorname{adj}}

\newcommand\scalemath[2]{\scalebox{#1}{\mbox{\ensuremath{\displaystyle 
#2}}}}

\mathcode`@="8000 
{\catcode`\@=\active\gdef@{\mkern1mu}}

  \newenvironment{keywords}{{\small \noindent\textbf{Key words.}}}{\medskip}
\newenvironment{AMScodes}{{\small \noindent\textbf{AMS subject classifications.}}}{\medskip}
\renewenvironment{abstract}{{\small \noindent\textbf{Abstract.}}}{\medskip}
	
\title{A note on the cross matrices\thanks{
		\footnotesize
	Version of March 31, 2025.
	}}
\author{Xiaobo Liu%
		\thanks{
			\footnotesize			
			 Max Planck Institute for Dynamics of Complex Technical Systems,
			 Magdeburg, 39106, Germany
			(\url{xliu@mpi-magdeburg.mpg.de}).
		}
}
\date{}

\makeatother

\begin{document}
\maketitle
\begin{abstract}	
A cross matrix $X$ can have nonzero elements located only on 
the main diagonal and the anti-diagonal, so that the sparsity pattern has 
the shape of a cross. It is shown that $X$
can be factorized into products of matrices that are at most rank-two perturbations to the identity matrix and can be symmetrically permuted to block 
diagonal form 
with $2\times 2$ diagonal blocks and, if $n$ is odd, a $1\times 1$ 
diagonal block. 
The permutation similarity implies that any well-defined analytic function of $X$ remains a cross matrix.
By exploiting these properties, explicit formulae for the determinant, inverse, and characteristic polynomial are derived.
It is also shown that the structure of cross matrix can be preserved under matrix factorizations, including the LU, QR, and SVD decompositions.
\end{abstract}


\begin{AMScodes}
	15B99, 15A23, 65F40  
\end{AMScodes}

\section{Introduction}

A matrix that has nonzero elements only on the main diagonal and anti-diagonal has the sparsity pattern of a cross shape.
We therefore refer to any matrix within this class as ``cross matrix'', which has the form
\begin{equation*}
	X = \scalemath{0.85}{\begin{bmatrix}
			x_{11} &            &    &         & x_{1,2k+1} \\
			&     \ddots&    & \reflectbox{$\ddots$}          & \\
			&      & x_{k+1,k+1} & 	             & \\
			&     \reflectbox{$\ddots$}&    & \ddots      & \\
			x_{2k+1,1} &      &        &  & x_{2k+1,2k+1} 
	\end{bmatrix}},
\end{equation*}
for $n=2k+1$, $k\in\mathbb{N^+}$, and, if $n=2k$, 
\begin{equation*}
	X = \scalemath{0.85}{\begin{bmatrix}
			x_{11} &        &       &    &             & x_{1,2k} \\
			&  \ddots&   & & \reflectbox{$\ddots$}   & \\
			&        & 	    & x_{kk} & x_{k,k+1}	&           & \\
			&        & 	    & x_{k+1,k} & x_{k+1,k+1}	&           & \\
			&     \reflectbox{$\ddots$}&   & & \ddots      & \\
			x_{2k,1} &     &        &    &  & x_{2k,2k} 
	\end{bmatrix}}.
\end{equation*}
It can be viewed as a generalization of a type of compound Jacobi rotation matrices used in the parallel Jacobi algorithm for symmetric eigenvalue problems \cite[sect.~5.8]{modi88}, \cite{same71}.
The set of all $n\times n$ cross matrices with complex entries can be defined as 
\begin{equation*}
	\mathcal{X}_n(\mathbb{C}) = \left\{
	X=(x_{ij})\in M_n(\mathbb{C})\colon 
	x_{ij}=0 \ \text{if}\ j\ne i\ \text{and} \ 
	j\ne n+1-i
	\right\},
\end{equation*}
where $M_n(\mathbb{C})$ is the set of all $n\times n$ matrices over the complex field.
This set, like many other important classes of square matrices, e.g., triangular matrices and diagonal matrices, possesses some special algebraic structure, as shown in Lemma~\ref{lem:ring}.
To this point, it is also interesting to note that 
the set of anti-diagonal matrices does not form a subring of $M_n(\mathbb{C})$ as the closure under multiplication fails to hold.

\begin{lemma}\label{lem:ring}
	$\mathcal{X}_n(\mathbb{C})$ is a subring of $M_n(\mathbb{C})$.
\end{lemma}
\begin{proof}
The satisfaction of the ring axioms is obvious except the closure under multiplication, which can be seen by splitting a cross matrix into the sum of a diagonal- and a anti-diagonal matrices.
\end{proof}

In this work we study properties of cross matrices.
We start with the derivation of two factorizations of cross matrices in Section~\ref{sect:mul-struc}, one as products of low-rank perturbation to the identity matrix and the other as a symmetrically permuted block-diagonal matrix that has diagonal blocks of size no more than two. These factorizations reveal the multiplicative structure of cross matrices and form the cornerstone of the discussion in Section~\ref{sect:det-inv-eig}, where we derive explicit formulae for the determinant, inverse, and characteristic polynomial. Finally, we show in Section~\ref{sect:mat-fac} the structure-preserving property of cross matrices under matrix factorizations that are commonly used in numerical linear algebra. 

Throughout the work we use $I_n$ to denote the identity matrix of order $n$, and with slight abuse of notation, $I_0$ denotes an identity matrix of order $0$ (an empty matrix). We use the MATLAB-style colon notation for indices, i.e., $1\colon k$ represents the index set $\{1,2,\dots,k\}$.
The $\diag$ operator returns a block diagonal matrix when there is a matrix input.

\section{Multiplicative structure}\label{sect:mul-struc}
The first result shows the cross matrices can be factorized into products of matrices that are at most rank-two perturbations to the identity matrix. It can be viewed as a generalization of the cross form of complete Jacobi rotation~\cite[sect.~5.5]{papa93}.

\begin{theorem}\label{thm:fac-ranktwo}
	The cross matrix $X=(x_{ij})$ can be factorized as
	\begin{equation*}
		X=\begin{cases}
			Y_1Y_2\cdots Y_k, & n = 2k,\\
			Y_1Y_2\cdots Y_k\cdot \diag(I_k,x_{k+1,k+1},I_k), 
			& n = 2k+1,
		\end{cases}
	\end{equation*}
	where
	\begin{equation}\label{eq:Y_i}
		Y_i := \scalemath{0.85}{\begin{bmatrix}
			I_{i-1}  &   &   &	&  \\
			         & x_{ii}   & 	& x_{i,n+1-i} &  \\
			&   &  I_{n-2i} &   &  \\
			&x_{n+1-i,i} &   &  x_{n+1-i,n+1-i}& \\
			  &  	&     &    & I_{i-1} \\
		\end{bmatrix}}
	\end{equation}
is the identity matrix with the intersection of its
$i$th and $(n+1-i)$th rows and columns replaced by that of $X$.	
\end{theorem}

\begin{proof}
First consider the case $n=2k+1$.
Observe that the $X$ can be factorized as 
\begin{equation*}
	X = Y_1 \cdot \diag(1, X(2\colon 2k,2\colon 2k), 1),
\end{equation*}
where the submatrix $X(2\colon 2k,2\colon 2k)$ of order $2k-1$ has the same 
cross-shaped form. It therefore
can be factorized in the same way as
\begin{equation*}
	X(2\colon 2k,2\colon 2k) = \begin{bmatrix}
		x_{22} &          & x_{2,2k} \\
		& I_{2k-3} &	\\
		x_{2k,2} &          & x_{2k,2k} \\
	\end{bmatrix}\cdot \diag(1, X(3\colon 2k-1,3\colon 2k-1), 1).
\end{equation*}
Again, the submatrix $X(3\colon 2k-1,3\colon 2k-1)$ of order $2k-3$ has the same 
cross-shaped form. This process is continued $k$ times,
until we obtain the factorization
\begin{equation*}
 X(k\colon k+2,k\colon k+2) =
	\begin{bmatrix}
		x_{kk} &          & x_{k,k+2} \\
		& 1 &	\\
		x_{k+2,k} &          & x_{k+2,k+2} \\
	\end{bmatrix}\cdot \diag(1, x_{k+1,k+1}, 1).
\end{equation*}
On completion of the process, the desired factorization is obtained by augmenting the factor matrices appropriately into a block-diagonal matrix with identity matrices. 

For the case $n=2k$, the factorization process terminates once 
we arrive at
\begin{equation*}
	X(k \colon k+1,k\colon k+1) =
	\begin{bmatrix}
		x_{kk}    & x_{k,k+1} \\
		x_{k+1,k} & x_{k+1,k+1} \\
	\end{bmatrix},
\end{equation*}
and the results can be proved in a similar way.
\end{proof}

Interestingly, one can show in a similar way to the proof of Theorem~\ref{thm:fac-ranktwo} that the product factorization of $X$ therein also holds in the reverse order. The key is to note that the two factors in the factorizations of $X, X(2:n-1,2:n-1), \dots, X(k:n+1-k,k:n+1-k)$ commute. The result is presented in the following corollary.

\begin{corollary}\label{cor:fac-ranktwo}
The cross matrix $X=(x_{ij})$ satisfies
\begin{equation*}
	X=\begin{cases}
		Y_kY_{k-1}\cdots Y_1, & n = 2k,\\
		\diag(I_k,x_{k+1,k+1},I_k) \cdot Y_kY_{k-1}\cdots Y_1, 
		& n = 2k+1,
	\end{cases}
\end{equation*}
for the same $Y_i$ defined in~\eqref{eq:Y_i}.
\end{corollary}

For the matrix $Y_i$ in~\eqref{eq:Y_i}, we have
\begin{equation}\label{eq:Y_i-tilde}
	P_iY_iP_i = 
	\diag\left(
	I_{i-1},\begin{bmatrix}
		x_{ii}   &    x_{i,n+1-i}   \\
		x_{n+1-i,i}  & x_{n+1-i,n+1-i}
	\end{bmatrix}, I_{n-i-1}
	\right)
=: \widetilde{Y}_i, \quad i=1\colon \floor{n/2},
\end{equation}
where $P_i$ is the \textit{elementary} permutation matrix (which is symmetric) formed by swapping the $(i+1)$th and
$(n+1-i)$th columns of the identity matrix $I_{n}$, and $\widetilde{Y}_i$ is block diagonal.
The relation~\eqref{eq:Y_i-tilde} can be used to prove the permutation similarity of $X$ to block diagonal matrices, as the next theorem shows.

\begin{theorem}\label{thm:blk-diag}
Any cross matrix	$X=(x_{ij})$ satisfies, for $m\in\mathbb{N}$,
\begin{equation*}
	PXP^T=\begin{cases}
		\diag(B_1,\dots,B_m,C_m,\dots,C_1), & n = 4m,\\
		\diag(B_1,\dots,B_m,x_{k+1,k+1},C_m,\dots,C_1), & n = 
		4m+1,\\
		\diag\left(B_1,\dots,B_m,
		\begin{bsmallmatrix}
			x_{kk} &  x_{k,k+1} \\
			x_{k+1,k} & x_{k+1,k+1}
		\end{bsmallmatrix}
		,C_m,\dots,C_1\right), & n = 4m+2, \\
		\diag(B_1,\dots,B_m,B_{m+1},x_{k+1,k+1},C_m,\dots,C_1), & n = 4m+3,
	\end{cases}
\end{equation*}
where
\begin{equation*}
B_{i}:= \begin{bmatrix}
	x_{2i-1,2i-1} & x_{2i-1,n+2-2i} \\
	x_{n+2-2i,2i-1} & x_{n+2-2i,n+2-2i}
\end{bmatrix}, \quad 
C_{i}:= \begin{bmatrix}
	x_{n+1-2i,n+1-2i} & x_{n+1-2i,2i} \\
	x_{2i,n+1-2i} & x_{2i,2i}
\end{bmatrix},
\end{equation*}
and $P$ is a permutation matrix given by
\begin{equation*}
P=\begin{cases}
	P_{2m+1}P_{2m-1}\cdots P_3P_1, & n = 4m+3, \\
	P_{2m-1}\cdots P_3P_1, & \text{otherwise}.
\end{cases}
\end{equation*}
\begin{proof}
Consider first the case $n=2k$.
We have, using the symmetric permutation matrices $P_i$ defined in~\eqref{eq:Y_i-tilde},
\begin{equation*}
	P_1XP_1 = \diag(B_1, X(3\colon 2k-2,3\colon 2k-2), C_1),
\end{equation*}
and subsequently, 
\begin{equation*}
	P_3P_1XP_1P_3 =
	\diag(B_1,  B_2,
	 X(5\colon 2k-4,5\colon 2k-4), C_2, C_1).
\end{equation*}
Then we apply a similarity transformation with $P_5$ to the resulting matrix to 
exchange the second and last columns and rows of the submatrix 
$X(5\colon 2k-4,5\colon 2k-4)$. This process is continued 
$m=\floor{n/4}$ times, and in total $m$ two-side permutations applied to $X$. We finally obtain
\begin{equation*}
	PXP^T=\begin{cases}
		\diag(B_1,\dots,B_m,C_m,\dots,C_1), & n = 4m,\\
		\diag(B_1,\dots,B_m,X(k\colon k+1,k\colon k+1),C_m,\dots,C_1), & n = 4m+2,
	\end{cases}
\end{equation*}
where $P:=P_{2m-1}\cdots P_3P_1$ is a permutation matrix because
$P_i=P_i^T$ implies $P^T=P_1^TP_3^T\cdots 
P_{2m-1}^T=P_1P_3\cdots P_{2m-1}$.

The proof for the case $n=2k+1$ is similar, and the slight difference lies in the final step. After the $m=\floor{n/4}$ two-side permutations applied to $X$, we obtain
\begin{equation*}
	PXP^T=\begin{cases}
		\diag(B_1,\dots,B_m,x_{k+1,k+1},C_m,\dots,C_1), & n = 
		4m+1,\\
		\diag(B_1,\dots,B_m,X(k\colon k+2,k\colon k+2),C_m,\dots,C_1), & n = 4m+3.
	\end{cases}
\end{equation*}
Therefore the right-hand side is already in the desired form for the case of $n=4m+1$. If $n = 4m+3$, then
one more transformation is needed for the central $3\times 3$ block $X(k\colon k+2,k\colon k+2)$, such that
\begin{equation*}
	P_{2m+1}(PXP^T)P_{2m+1}^T = 
	\diag(B_1,\dots,B_m,B_{m+1},x_{k+1,k+1},C_m,\dots,C_1),
\end{equation*}
where $P_{2m+1}P$ clearly remains a permutation matrix.
\end{proof}
\end{theorem}

We have shown the cross matrix $X$ is unitarily similar to  
a block diagonal matrix with $2\times 2$ blocks and a $1\times 1$ block if $n$ is odd. 
The following corollary follows immediately.

\begin{corollary}\label{cor:f(X)}
Any polynomial $p(X)$ of a cross matrix $X$ remains cross matrix. Furthermore, $f(X)$ is a cross matrix for any analytic function $f$ defined on the spectrum of $X$.
\end{corollary}
\begin{proof}	
For the first part, note that any nonnegative integer powers of $X$ have the same shape.
This follows from the symmetrically permuted block diagonal form in Theorem~\ref{thm:blk-diag}, as it is essentially powering up this block diagonal form and then undoing the permutation.
The second part holds since analytic functions of a matrix can be expressed as a polynomial in the matrix~\cite[Def.~1.4]{high:FM}.
\end{proof}

See~\cite[Prob.~10.14]{high:FM} for an interesting example of the matrix exponential of $X$ that is parameter dependent.

\section{Determinant, inverse, and eigenvalues}\label{sect:det-inv-eig}
It follows from~\eqref{eq:Y_i-tilde} that
\begin{align}\label{Yi-det}
	\det(Y_i) &
	=\det(P_i)\det(\widetilde{Y}_i)\det(P_i) 
	 =\det\left(
	\begin{bmatrix}
		x_{ii}   &    x_{i,n+1-i}       \\
		x_{n+1-i,i} 	 & x_{n+1-i,n+1-i} 
	\end{bmatrix}
	\right) \nonumber \\
	& = x_{ii}x_{n+1-i,n+1-i}-x_{i,n+1-i}x_{n+1-i,i}, \quad i=1\colon \floor{n/2}.
\end{align}
Since $Y_{i}$ is in the form of the identity plus a rank-two
matrix, this formulae could be alternatively derived by using
\begin{equation*}
	\det(I_n+vx^*+wy^*) = (1+v^*x)(1+w^*y)-(v^*y)(w^*x),\quad 
	v, w, x, y\in\C^n,
\end{equation*}
which is a special case of the Weinstein–Aronszajn identity~\cite[Prob.~1.2.17]{hojo13}.
By applying the formula~\eqref{Yi-det} to the factor matrices $Y_i$ in Theorem~\ref{thm:fac-ranktwo}, we obtain the formula for the determinant of cross matrices.
\begin{theorem}\label{thm:X-det}
	For a cross matrix $X=(x_{ij})$,
	\begin{equation*}
		\det(X) =
		\begin{cases}
			\prod_{i=1}^{k} (x_{ii}x_{2k+1-i,2k+1-i}-x_{i,2k+1-i}x_{2k+1-i,i}),
			& n = 2k,\\
			x_{k+1,k+1}\prod_{i=1}^{k}
			(x_{ii}x_{2k+2-i,2k+2-i}-x_{i,2k+2-i}x_{2k+2-i,i}), & n = 2k+1.
		\end{cases}
	\end{equation*}
\end{theorem}

Based on the determinant formula of Theorem~\ref{thm:X-det}, a convenient approach to find the inverse is via~\cite[sect.~0.8.2]{hojo13}
\begin{equation}\label{eq:X-inv-adj}
	X^{-1} = \adj(X)/\det(X),\quad \det(X)\ne 0,
\end{equation}
where $\adj(X) = \big((-1)^{i+j}\det(X_{ji})\big)$ is the adjugate of $X$, and  $X_{ji}$ denotes the submatrix of $X$ obtained by deleting row $j$ 
and column $i$. 
The next lemma provides an explicit expression for $\det(X_{ji})$.

\begin{lemma}\label{lem：minor-submat}
Define 
$\alpha_j := x_{jj}x_{n+1-j,n+1-j}-x_{j,n+1-j}x_{n+1-j,j}$.
The determinant of the $(j,i)$-th minor submatrix $X_{ji}$ of a cross matrix $X$ is given by
\begin{equation*}
	\det(X_{ji}) =
	\begin{cases}	
	 	x_{n+1-j,n+1-j}\cdot \det(X) / \alpha_j, 
		& i = j,\ i+j\ne n+1, \\
		(-1)^{n} x_{ij}\cdot \det(X) / \alpha_j,
		& i = n+1-j, \ i\ne j, \\
		\det(X) / x_{k+1,k+1}, 
		& i=j=k+1,\ n=2k+1, \\
		0, 
		& \text{otherwise}.
	\end{cases}
\end{equation*}
\end{lemma}
\begin{proof}
	The key observation is that the submatrix $X_{ji}$ is singular when $j\ne i$ and $j\ne n+1-i$, in which case the rows with these two indices are linearly dependent when the $i$th column is removed.
	A immediate consequence is $\adj(X)$ remains in the cross shape (which is expected from Corollary~\ref{cor:f(X)} and \eqref{eq:X-inv-adj}).
	
	When both $j= i$ and $j= n+1-i$ are satisfied, it is equivalent to $n=2k+1$ with $i=j=k+1$. In this case $X_{ji}$ is already a cross matrix with dimension $2k$, and so $\det(X_{ji}) = \det(X) / x_{k+1,k+1}$. 
	
	If only one of the conditions $j= i$ and $j= n+1-i$ holds, 
	$X_{ji}$ will become a cross matrix if the $(n+1-j)$th row and $(n+1-i)$th column (of the original matrix $X$) are further deleted.  The latter case with $n=2k$,
	is illustrated in~\eqref{eq:X_ji}, and the other cases are similar.
\begin{equation}\label{eq:X_ji}
	X_{ji} = 
	\scalemath{0.8}{
	\begin{bNiceMatrix}[first-row,first-col, name=Matrix, cell-space-limits=1pt]
	&  &  &	&  &  &  &  & i &  &\\ 
	&x_{11} &  & & & & &  & \phantom{1} &  & x_{1n} \\
	& & \ddots & & & & &  &  & \reflectbox{$\ddots$} & \\
j	&\phantom{1} & & x_{jj} &  &  & & & x_{j,i} & & \phantom{1} \\
	& &  & & \ddots & & & \reflectbox{$\ddots$} &  &  & \\
	& & & & & x_{kk}  & x_{k,k+1} &  & & & \\
	& & & & & x_{k+1,k} & x_{k+1,k+1} & & &  &  \\
	& &  & & \reflectbox{$\ddots$} & & & \ddots &  &  & \\
	& & & \Block[draw]{}{x_{n+1-j,n+1-i}} & & & & & x_{n+1-j,i} & &  \\
	& & \reflectbox{$\ddots$} & & & & &  &  & \ddots & \\
	& x_{n1} &  & & & & &  & \phantom{1} &  & x_{nn} \\
	\CodeAfter
	\tikz \draw [black, thick] ([yshift=1mm]Matrix-1-8.north) -- ([yshift=-1mm]Matrix-10-8.south);
	\tikz \draw [black, thick] ([xshift=-2mm,yshift=-.5mm]Matrix-3-1.west) -- ([xshift=2mm,yshift=-.5mm]Matrix-3-10.east);
	\end{bNiceMatrix}},\quad i+j = n+1. 
\end{equation}
	Therefore, in this case the determinant of $X_{ji}$ is obtained by expanding along the $(n+1-j)$th row or the $(n+1-i)$th column and then invoking Theorem~\ref{thm:X-det}.
	Note that the minor corresponding to the element $x_{n+1-j,n+1-i}$ of $X_{ji}$ is equal to $\det(X) / \alpha_j$, and that
	the values of $i$ and $j$ will decide the sum of the row and column indices of $x_{n+1-j,n+1-i}$ in $X_{ji}$, which determines the sign of $\det(X_{ji})$. 
	This index sum will shift by $1$ only when $x_{n+1-j,n+1-i}$ locates on the anti-diagonal. 
	The proof completes by examining the parity of the index sum.
\end{proof}

By applying Lemma~\ref{lem：minor-submat} to the adjugate in~\eqref{eq:X-inv-adj}, we obtain the explicit formula for the inverse of cross matrices.

\begin{theorem}
	The inverse of a cross matrix remains the same shape, and its nonzero elements are given by
	\begin{equation*}
		(X^{-1})_{ij} =
		\begin{cases}
			x_{n+1-j,n+1-j} / \alpha_j,
			& i = j,\ i+j\ne n+1, \\
			-x_{ij}  / \alpha_j, 
			& i = n+1-j, \ i\ne j, \\
			1 / x_{k+1,k+1}, 
			& i=j=k+1,\ n=2k+1.
		\end{cases}
	\end{equation*}
	where $\alpha_j = x_{jj}x_{n+1-j,n+1-j}-x_{j,n+1-j}x_{n+1-j,j}$. 
\end{theorem}

In a sense the same form of $X^{-1}$ as $X$ makes the cross matrix seem rather trivial.
But not all methods of interest are invariant under such a 
transformation; for example, it changes the pivots for LU factorization with pivoting.

An alternative approach to get the $X^{-1}$ is to show the $Y_i^{-1}$ has the same structure as the $Y_i$ in~\eqref{eq:Y_i} and then make use of the factorizations of $X$ in Theorem~\ref{thm:fac-ranktwo} and Corollary~\ref{cor:fac-ranktwo}.

For the eigenvalues of $X$, we consider its characteristic polynomial, which is simply $\det(X-\lambda I)$. 
Using the determinant formula in Theorem~\ref{thm:X-det}, we obtain
\begin{equation}\label{eq:eig-val}
	\det(X) =
	\scalemath{0.83}{
	\begin{cases}
		\prod_{i=1}^{k} \left( (x_{ii}-\lambda)(x_{2k+1-i,2k+1-i}-\lambda) -x_{i,2k+1-i}x_{2k+1-i,i}\right),
		& n = 2k,\\
		(x_{k+1,k+1}-\lambda) \prod_{i=1}^{k} \left(
		(x_{ii}-\lambda)(x_{2k+2-i,2k+2-i}-\lambda)-x_{i,2k+2-i}x_{2k+2-i,i}\right), & n = 2k+1.
	\end{cases}}
\end{equation}
Hence the eigenvalues of $X$ are the roots of these $\floor{n/2}$ scalar quadratic equations.

\section{Matrix factorizations}\label{sect:mat-fac}
In this section we show that factors of the commonly used factorizations of the cross matrices preserve the structure, which is a property  possessed by only few special classes of matrices, e.g., diagonal matrices. 

We start our discussion with the LU factorization, followed by its implication for the Cholesky factorization when $X$ is symmetric positive definite.

\begin{theorem}
	Let $X$ be a cross matrix with $x_{ii}\ne 0$, $i=1\colon \floor{n/2}$. Then the LU decomposition without pivoting $X=LU$ produces cross-shaped factor matrices $L$ and $U$. 
	In particular, the nonzero elements in the strictly lower triangular part of $L$ are 
	$\ell_{n+1-i,i} = x_{n+1-i,i} / x_{ii}, i = 1\colon \floor{n/2}$,
	and $U$ has the same elements as the upper triangular part of $X$, except 
	$u_{ii} = \alpha_{i} / x_{n+1-i,n+1-i}, i = \ceil{n/2}+1\colon n$,
	where $\alpha_{i} =x_{ii}x_{n+1-i,n+1-i}-x_{i,n+1-i}x_{n+1-i,i}$.
\end{theorem}
\begin{proof}
	There are in total $\floor{n/2}$ stages in the LU factorization to zero the $\floor{n/2}$ anti-diagonal elements in the strictly lower triangular part of $X$. At the $i$th stage, the only nonzero multiplier is $\ell_{n+1-i,i} = x_{n+1-i,i} / x_{ii}$ due to the cross shape of $X$.
	The explicit form of $U$ is then easily obtained by $U=L^{-1}X$. In particular, the first $\ceil{n/2}$ rows of $U$ remain the same as those of $X$ since these rows of $X$ are in upper triangular form.
\end{proof}

\begin{corollary}
	The Cholesky factorization $X=R^TR$ of a symmetric positive definite cross matrix produces cross Cholesky factor $R$.
\end{corollary}
\begin{proof}
	For symmetric positive definite $X$ the LU factorization $X=LU$ is unique with no pivoting required~\cite[sect.~10.1]{high:ASNA2}.
	It remains to check $R=D^{-1}U=DL^T$, where $D$ is a diagonal matrix with $d_{ii}=\sqrt{u_{ii}}$, and it  follows from the symmetry of $X$.
\end{proof}

For the QR factorization it is most convenient to prove the result using Givens rotations due to the special structure of the cross matrix $X$, though alternatively it can be proved by using the permutation similarity of $X$ (Theorem~\ref{thm:blk-diag}) and considering QR factorization of the diagonal blocks.

\begin{theorem}
The Givens QR factorization $X=QR$ of a cross matrix $X$ produces cross factor matrices $Q$ and $R$. 
\end{theorem}
\begin{proof}
	For each column $j$ of $X$, the only nonzero element that needs to be zeroed out is on the anti-diagonal, and in total there are $k=\floor{n/2}$ such nonzero elements for $j=1\colon k$.
	Each element elimination can be achieved by applying a Givens rotation~\cite[sect.~19.6]{high:ASNA2} to rows $j$ and $n+1-j$, and therefore, the $j$th Givens rotation matrix $G_j$ (and its transpose) has the same structure as the $Y_j$ in~\eqref{eq:Y_i}.
	We have $G_kG_{k-1}\cdots G_1 X = R$, where $R$ is upper triangular by construction. 
	From Theorem~\ref{thm:fac-ranktwo} we know 
	$Q:= G_1^T\cdots G_{k-1}^TG_k^T$ is a cross matrix.
	The cross form of $R$ follows from the relation $R=Q^{-1}X$, together with Corollary~\ref{cor:f(X)} and Lemma~\ref{lem:ring}.
\end{proof}

Theorem~\ref{thm:blk-diag} has shown that cross matrices can be block-diagonalized with $2\times 2$ diagonal blocks and a $1\times 1$ block if $n$ is odd. The next result shows a diagonalizable cross matrix can be diagonalized by cross matrices.

\begin{theorem}\label{thm:spec-decom}
	If $X$ is a diagonalizable cross matrix, then there exits a spectral decomposition $X=VDV^{-1}$ such that $V$ is cross matrix.
\end{theorem}
\begin{proof}
	First consider the case $n=2k$.
	The eigenvector equation $Xv = \lambda v$ for a given eigenvalue $\lambda$ consists of $n$ scalar equations of the form
	\begin{equation*}
		x_{ii}v_i + x_{i,n+1-i} v_{n+1-i} = \lambda v_i,\quad 
		1 \le i \le n.
	\end{equation*}
	each of which only involves $v_i$ and $v_{n+1-i}$.
	This means each eigenvector equation can be decoupled into $k$ homogeneous $2\times 2$ systems
	\begin{equation}\label{eq:eigen-system-2by2}
		\begin{bmatrix}
			x_{ii}-\lambda & x_{i,2k+1-i} \\
			x_{2k+1-i,i} & x_{2k+1-i,2k+1-i} -\lambda
		\end{bmatrix} \begin{bmatrix}
			v_i \\ v_{2k+1-i} 
		\end{bmatrix} = \begin{bmatrix}
		0 \\ 0
		\end{bmatrix},\quad i= 1\colon k.
	\end{equation}
	Suppose $(\lambda_i,\lambda_{2k+1-i})$ is a pair of eigenvalues obtained by solving the scalar quadratic equation from~\eqref{eq:eig-val},	$(x_{ii}-\lambda)(x_{2k+1-i,2k+1-i}-\lambda) - 
		x_{i,2k+1-i}x_{2k+1-i,i} = 0$,
	so it is the eigenvalue pair such that the coefficient matrix of~\eqref{eq:eigen-system-2by2} becomes singular. This implies that the two linearly independent eigenvectors corresponding to $(\lambda_i,\lambda_{2k+1-i})$ can be chosen to have nonzero values in the $i$th and $(2k+1-i)$th positions and zeros everywhere else.
	By construction, each eigenvector has nonzero values only in two positions $(i, 2k+1-i)$, $i=1\colon 2k$, and therefore the full eigenvector matrix $V$ forms a cross matrix.
	
	The case $n=2k+1$ is similar; the main difference is that the eigenvector equation is now decomposed into $k$ linear systems of size $2\times 2$ plus the uncoupled equation $x_{k+1,k+1}v_{k+1} = \lambda v_{k+1}$. The latter equation means $x_{k+1,k+1}$ is an eigenvalue of $X$, and its corresponding eigenvector can be chosen as the $(k+1)$th standard basis vector, which is linearly independent to the other $2k$ eigenvectors by construction.	
\end{proof}

The proof of Theorem~\ref{thm:spec-decom} reveals that eigenspace of diagonalizable cross matrices is a direct sum of two-dimensional subspaces and a one-dimensional subspace if $n$ is odd. This is a consequence of the intrinsic pairing symmetry in the eigenvector equation of cross matrices. 
Indeed, this eigenspace structure can also be seen from Theorem~\ref{thm:blk-diag}, which implies that the diagonalizability of $X$ depends on that of each diagonal blocks therein.

\begin{theorem}\label{thm:svd}
	For a cross matrix $X$, there exists a SVD, $X=U\Sigma V^*$, such that both $U$ and $V$ are cross matrices.
\end{theorem}
\begin{proof}
	Suppose $X$ is converted into the one of the block diagonal forms in Theorem~\ref{thm:blk-diag}, say,
	$PXP^T = \diag(B_1,\dots,B_m,C_m,\dots,C_1)$.
	If the SVDs of each of the $2\times 2$ blocks are computed such that
	$B_i=U_i \Sigma_i V_i^*$ and
	$C_i=U_{2m+1-i}\Sigma_{2m+1-i} V_{2m+1-i}^*$, $i=1\colon m$, then we have
	\begin{equation*}
		PXP^T = \diag(U_1,\dots,U_{2m})\cdot 
		\diag(\Sigma_1,\dots,\Sigma_{2m})\cdot 
		\diag(V_1^*,\dots,V_{2m}^*).
	\end{equation*}
	The desired SVD of $X$ is obtained by taking 
	$U = P^T \diag(U_1,\dots,U_{2m}) P$ and 
	$V = P^T\diag(V_1,\dots,V_{2m}) P$, which are unitary by construction and obviously have the same sparsity structure as $X$.
	The singular values of $X$ are contained in the diagonal matrix
	$\Sigma = P^T\diag(\Sigma_1,\dots,\Sigma_{2m}) P$, albeit not necessarily in descending order. 
	
	The proofs for the other three cases are similar; for odd $n$ the left- and right singular matrices of the one-dimensional diagonal block should be taken as the identity matrix with appropriate sign, depending on the sign of $x_{k+1}$.
	\end{proof}

	The polar decomposition with cross polar factors can be obtained via the SVD discussed in Theorem~\ref{thm:svd}, as the next result shows.
	
	\begin{theorem}
	If $X$ is nonsingular, then the polar decomposition $X=UH$ is unique and both the unitary polar factor $U$ and the Hermitian polar factor $H$ are cross matrices; 
	otherwise, there exists a polar decomposition $X=UH$ with cross factors $U$ and $H$.
	\end{theorem}
	\begin{proof}
	The Hermitian polar factor $H$ is given by $H=(X^*X)^{1/2}$~\cite[Thm.~8.1]{high:FM}, whose cross sparsity pattern follows from Lemma~\ref{lem:ring} and then Theorem~\ref{cor:f(X)}. 	
			
	If $X$ is nonsingular, the unitary polar factor $U$ is uniquely determined by $U = XH^{-1}$, which is a cross matrix. If $X$ is singular, there are infinitely many unitary polar factors but the one with a cross shape is obtained by multiplying the left- and right singular matrices of the SVD in Theorem~\ref{thm:svd}.
	\end{proof}


\section*{Acknowledgment}
This work was initiated when the author was a research associate at The University of Manchester, during which time the valuable comments from the late Prof.~Nick Higham on a draft manuscript are gratefully acknowledged. 

\bibliographystyle{myplain2-doi}
 \bibliography{notebib}

\end{document}